\documentclass[a4paper]{article}
\usepackage{amsfonts,amsmath,amssymb,amsthm}
\usepackage{color} 
\usepackage{graphicx}

\numberwithin{equation}{section} 

\newtheorem{lem}{Lemma}[section]
\newtheorem{theo}[lem]{Theorem}
\newtheorem{cor}[lem]{Corollary}
\newtheorem{prop}[lem]{Proposition}
\newtheorem{rk}[lem]{Remark}

\newtheorem{hyp}[lem]{Hypothesis}
 \newcommand{\n}{n}

\newcommand{\hausmoins}{\mathcal{H}^{\n-1}}
\newcommand{\Sn}{\bS^{\n-1}}
\newcommand{\bS}{\mathbb S}

\def\Om{\Omega}

\def\R{\mathbb{R}}

\def\N{\mathbb{N}}

\def\F{\mathcal{F}}

\def\LM#1{\hbox{\vrule width.2pt \vbox to#1pt{\vfill \hrule width#1pt
height.2pt}}}
\def\LL{{\mathchoice {\>\LM7\>}{\>\LM7\>}{\,\LM5\,}{\,\LM{3.35}\,}}}
 \newcommand{\restr}{{\LL}}
\renewcommand{\phi}{\varphi}
\def\p{\phi}

\def\car#1{\1_{#1}}
\def\dist{\textup{dist}}
\def\ds{\displaystyle}
\def\ov{\overline}
\def\H{\mathcal{H}}
\def\HH{\mathcal{H}^{n-1}}
\def\1{\mathbf{1}}
\def\loc{\mathrm{loc}}

\def\XXint#1#2#3{{\setbox0=\hbox{$#1{#2#3}{\int}$ }
\vcenter{\hbox{$#2#3$ }}\kern-.57\wd0}}

\def\eps{\varepsilon}
\newcommand{\sprod}[2]{\langle #1, #2 \rangle}

\def\Anx{A_{\nu x}}

\def\sigbrb{\ov{\sigma}_{\beta}}
\DeclareMathOperator*{\argmin}{\arg\!\min}
\def\Wr{W}
\def\Eund{\mathcal{G}_{\beta}}
\def\lt{\left}
\def\rt{\right}
\def\car#1{\textbf{1}_{#1}}
 
  \usepackage{pdfsync}
   \usepackage[pdftex]{hyperref}
  \hypersetup{colorlinks=true,citecolor=blue, filecolor=red, linkcolor=blue, urlcolor=blue}

\def\Feb{\F_{\eps,\beta}}
\def\Fb{\F_{\beta}} 
\def\be{\begin{equation}}
\def\ee{\end{equation}}
\def\ben{\begin{equation*}}
\def\een{\end{equation*}}
\def\ba{\begin{eqnarray}}
\def\ea{\end{eqnarray}}
\def\ban{\begin{eqnarray*}}
\def\ean{\end{eqnarray*}}
\def\D{\mathcal{D}}
\def\sigb{\sigma_{\beta}}
\def\1{\mathbf{1}}

\begin{document}

\title{Sharp interface limit for two components Bose-Einstein condensates}
\author{
        M. Goldman
        \footnote{Max Planck Institute for Mathematics in the Sciences,  Inselstrasse 22, 04103, Leipzig, Germany, email: goldman@mis.mpg.de, funded by a Von Humboldt PostDoc fellowship}
        \and J. Royo-Letelier
        \footnote{Institute of Science and Technology Austria (IST Austria), Am Campus 1, 3400 Klosterneuburg, Austria, email:jimena.royo-letelier@ist.ac.at}
}
\date{}

\maketitle

\begin{abstract}
We study a double Cahn-Hilliard type functional related to the Gross-Pitaevskii energy of two-components Bose-Einstein condensates.
In the case of large but same order intercomponent and intracomponent coupling strengths, we prove $\Gamma$-convergence to a perimeter minimisation functional with
 an inhomogeneous surface tension. We study the asymptotic behavior of the surface tension as the ratio between the intercomponent and intracomponent 
coupling strengths becomes very small or very large and obtain good agreement with the physical literature. We obtain as a consequence, symmetry breaking of the minimisers for the harmonic potential. 
\end{abstract}

 \section{Introduction}
For $V$  a given   trapping potential (see Hypothesis \ref{hypo} below for more precise requirement) and a fixed constant $\eps>0$ let $\eta_\eps$ 
 be the (unique) positive minimiser  of the Gross-Pitaevskii functional

\be \label{defGP1}
	E_\eps(\eta): = \frac{1}{2} \int_{\R^n} |\nabla \eta|^2 + {\frac{1}{\eps^2}} V   |\eta|^2 + \frac1{2\eps^2}   |\eta|^4  \, dx\,, 
\ee 
under the constraint $ \|\eta \|_2=1$, where $\|\eta\|_2$ denotes the $L^2(\R^n)$ norm of $\eta$. We then consider for $\beta$, $\alpha_1$ and $\alpha_2$ positive constants, with $\alpha_1+\alpha_2=1$, the double Cahn-Hilliard type functional    

\be \label{defFeb}
	\Feb(v,\p) :=    \frac{1}{2}\int_{\R^n} \eta_\eps^2 |\nabla v|^2 +  \frac1{2\eps^2} \eta_\eps^4 (1-v^2)^2 +  \frac14   \eta_\eps^2 v^2  \, |\nabla \p|^2 + \frac1{4 \eps^2} \beta  \eta_\eps^4 v^4 \,  \sin^2 \p  \, dx    \,,
\ee 
 under the mass constraints 
\be \label{massvphi}
	\int_{\R^n}   \eta_\eps^2 v^2\, dx = \alpha_1 + \alpha_2 =1 \qquad \text{ and }  \qquad \int_{\R^n}  \eta_\eps^2 v^2 \cos \p \, dx= \alpha_1 - \alpha_2 \,,
\ee 
 and study its behavior when the parameter $\eps$ tends to zero. \\

This functional arises in the description of two-components Bose-Einstein condensates with equal intracomponent coupling strengths (see Section \ref{Sec:derenergy}).
 The parameter $\frac{1}{\eps^2}$ represents the intracomponent coupling strength whereas  $1+\beta$ is  the ratio between the intercomponent and intracomponent coupling strengths. \\

The Gross-Pitaevskii functional \eqref{defGP1}, which describes the energy of a single component condensate with density $|\eta_\eps|^2$, has been extensively studied in the literature  \cite{AfLivre,AftJR,IgnMil,KarSour}.  
As $\eps$ goes to zero, $\eta_\eps$ converges to the Thomas-Fermi profile $\sqrt \rho$, given by

\be \label{defrhointro}
	\rho(x) := ( \lambda^2-V(x))_+ \, 
\ee

with $\lambda$ determined by the constraint $\ds \int_{\R^n}\rho \, dx=1$.  The support of $\rho$ is a domain denoted by $\D$ and corresponds to the region where the density of the single component condensate does not vanish as $\eps \to 0$.\\

The main result of the paper is the $\Gamma$-convergence \cite{braidesbook,braides} of  $\eps \Feb$ to  a perimeter minimisation problem with an inhomogeneous surface tension $\sigb$, defined in $\D$ by $\sigb(x):=\rho(x)^{3/2}\sigbrb$ with 

\begin{equation} \label{defsurftenintro}
 	\sigbrb:=\inf\left\{ \frac{1}{2} \int_{-\infty}^{+\infty}   v'^2 + \frac{1}{2} \left(1-v^2\right)^2+\frac{1}{4} v^2 \p'^2 +\frac{\beta}{4} v^4  \sin^2 \p \, dt \,  : \right. 
  \left.  \lim_{-\infty} \p= 0 \textrm{ and } \, \lim_{+\infty} \p =\pi\right\}   \,,
\end{equation}
where in the infimum, the function $v$ (respectively $\p$) denotes a function from $\R$ to $[0,1]$ (respectively from $\R$ to $[0,\pi]$).

\begin{theo} \textbf{ \textnormal ($\Gamma$-convergence)} \label{gammavctheo}
	Let $\beta>0$ be fixed. Under the Hypothesis \ref{hypo}, the $\Gamma$-limit in $L^1_{loc}(\D) \times L^1_{loc}(\D) $ as $\eps \to 0$ of $\eps \Feb$ with mass constraint \eqref{massvphi} is given by the functional $\Fb$ defined as 
	
\be \label{defFbeta}
	\Fb(v,\p) :=  \left\{   \begin{array}{cl}
	\ds \int_{\D} \frac{\sigb}{\pi}  |D\p|  & \text{ if $v=1$ a.e. in $\D$ and $\p\in BV_{loc}(\D;\{0,\pi\})$ }    \\ \\ 
	+\infty & \text{ otherwise}     \,,
	\end{array}\right. 
\ee
with mass constraint 
\be \label{mcfinal}
  \int_{\R^n} \rho \cos \p \, dx= \alpha_1-\alpha_2 \,.
\ee
  \end{theo}

Since $\Fb$ is finite only for $v=1$, we will denote by $\Fb(\p):=\Fb(1,\p)$.  It is worth noticing that since
 $\Fb(\p)=\ds \frac{\sigbrb}{\pi} \int_{\D} \rho^{3/2}  |D\p|$, the minimizers of $\Fb$ do not depend on $\beta$. This fact, which is quite peculiar to BEC interfaces, was already well known in the physics literature (see \cite{BVS}). The functional $\Feb$ shares at the same time some features with the celebrated Ambrosio-Tortorelli functional
 which is approximating the  Mumford-Shah functional (see \cite{AmbTort,AmbFusPa}), and some other with functionals appearing in the study of phase transitions such as the Modica-Mortola energy \cite{modica} 
(also known as Cahn-Hilliard or Allen-Cahn functional) or more general weighted  functionals \cite{Bou} (see also \cite{braides,braidesbook}). Indeed, $\Feb$ consists of the sum of two singularly perturbed, weighted double-well potentials which are coupled together.
 As in \cite{AmbTort,braides,AftRoyo,Bou} our proof is based on the slicing method described in Section \ref{sub:slicing}.\\

In experiments realised with two-components Bose-Einstein condensates \cite{McCarronETall,HaJILA,PaJILA}, the segregation of the components is observed for large values of the intercomponent 
coupling strengths. This has also been supported by numerical simulations in respectively, one (\cite{KasYaTsu}), two (\cite{KaTsuUe,MaAf}) and three (\cite{OhSten})  space dimensions.
In our setting, at the level of $\F_{\eps,\beta}$ this means that for large values of $\beta$, $\phi$ takes approximately only values $0$ and $\pi$ while $v$ is almost everywhere close to one. 
 Moreover, for the harmonic potential $V=|x|^2$ in dimension $n=2$ \cite{McCarronETall,MaAf}, one also observes a symmetry breaking in the sense that while  $V$ is radially symmetric, the support of each component (which correspond respectively to $ A :=\{\p=\pi\}$  and $\D \setminus A =\{\p=0\} $) are not.
 The  numerical simulations also show that near $\partial A$, the function $v$ is close to a small positive constant.  For $\beta<0$ the two components do not segregate and their densities are both proportional to $\rho$.  \\

We mention that segregation of two-components condensates has been widely studied for bounded intracomponent coupling strengths and large intercomponent 
coupling strength. In \cite{JRL} segregation and symmetry breaking is proven in $\R^2$ for small intracomponent coupling strengths. In \cite{WeWe1}, working on a bounded domain of $\R^2$ and taking the trapping potential $V$ to be zero, 
the authors show segregation and local uniform convergence of the two components. In \cite{conterrver1,NoTaTeVe} the 
regularity of $\partial A$ is studied for the same model. The profile of the components near $\partial A$ is analysed  in \cite{BereLinWeiZhao,BereTeraWaWei}. \\

In \cite{AftRoyo} the functional $\Feb$ is studied for $n=2$ when $\beta$ goes to $+\infty$ as $\eps$ tends to zero.
 The authors also prove $\Gamma$-convergence
 to a perimeter minimisation problem with an inhomogeneous surface tension. The main difference with our setting is that for $\beta\to +\infty$, the limiting energy is 
given by the first two terms of $\eps \Feb$ while the  last two terms go to zero as $\eps \to 0$. This leads to some decoupling of the energy which allows to compute explicitly the limiting surface tension.
%
%
%
%
In our case, all the terms in the energy $\eps \Feb$ are of the same order so that the surface tension is given by the one dimensional optimal transition problem  (\ref{defsurftenintro}). 
Thus,  we need to precisely analyse the behavior of $\sigbrb$ and of the associated optimal profile. We prove existence and qualitative properties of 
minimisers of $\sigb$, an equipartition of the energy and compare our results with the physical literature \cite{BVS,barankov,Aochui,Timmermans,Mazets}. In particular, we prove that minimisers $(v,\p)$ of  $\sigbrb$ satisfy $\inf v= m(\beta)>0$,  
 as was expected from numerical simulations. We remark that we are unable to prove uniqueness of the optimal profile. 
We study the asymptotic behavior of  $\sigbrb$ when $\beta$ tends to zero or infinity.
 On the one hand, we prove that when $\beta \to +\infty$, we recover the functional derived in \cite{AftRoyo}.  We show that in this regime, $\sigbrb \simeq \beta^{-1/4}$ as predicted by formal asymptotic expansions \cite{BVS}.
 This estimate follows from the fact that $m(\beta)\sim \beta^{-1/4}$ (see Proposition \ref{decayinf}). This fact is related to some open questions raised in \cite{BereLinWeiZhao} (see also the discussion in \cite{AftRoyo}). 
On the other hand, we show that as expected from \cite{Aochui,Timmermans,Mazets,barankov}, $\sigbrb\simeq \sqrt{\beta}$ when $\beta$ goes to zero. The fact that $\sigbrb$ vanishes in this limit, reflects the non segregation of the two components.
Finally, in Proposition  \ref{brissym}, we extend the symmetry breaking result for minimizers of $\mathcal{F}_\infty$ (for the harmonic potential $V=|x|^2$) 
obtained in \cite{AftRoyo} to space dimensions $n=1$ and $n=3$. We notice that since the minimizers of $\Fb$ coincide with the minimizers of $\mathcal{F}_\infty$, this symmetry breaking result 
 extends to any $\beta>0$ and by $\Gamma-$convergence to minimizers of the original functional $\Feb$ for $\eps$ small enough. \\

The paper is organised as follows:   in Section \ref{Sec:notation} we recall the definition and main properties of functions of bounded variation and the slicing method. In Section \ref{Sec:derenergy}, we explain how the functional $\Feb$ arises from the coupled Gross-Pitaevskii energy of 
a two-components Bose-Einstein condensate. In Section \ref{Sec:surftens} we study the variational problem $\eqref{defFbeta}$   and
 $\beta>0$, and prove existence and qualitative properties of minimisers. In Section \ref{Sec:gammacv}, we prove our main  $\Gamma$-convergence theorem. 
Finally, in Section \ref{Sec:surfasympt} we analyse the asymptotic behavior of $\sigb$ when $\beta$ tends to zero or infinity and prove as a consequence symmetry breaking of the minimisers.

\section{Notation} \label{Sec:notation}
For $x\in \R^n$ and $r>0$, we denote by $B_r(x)$ the ball of radius $r$ centered at $x$ and simply write $B_r$ when $x=0$. We let $\Sn$ be the unit sphere in $\R^n$ and for $k\in[0;n]$, we denote by $\H^k$ the $k-$dimensional Hausdorff measure.
Given a set $E\subset \R^n$, we let $\1_E$ be the characteristic function of the set $E$. The letters, $c, C$ denote universal constants which can vary from line to line. We also make use of the usual $o$ and $O$ notation. For $a$ and $b$ real numbers we let $a\wedge b:=\min(a,b)$ and $a\vee b:= \max(a,b)$.
 Throughout the paper, with a small abuse of language, we call sequence a family $(u_\eps)$ of functions 
labeled by a continuous parameter $\eps\in (0,1]$. A subsequence 
of $(u_\eps)$ is any sequence $(u_{\eps_k})$ such that $\eps_k
\to 0$ as $k \to +\infty$. We mention that  $\rho$ will denote a positive constants in Sections \ref{Sec:surftens} and \ref{Sec:surfasympt},  while in the rest of the paper it will be the function given in (\ref{defrhointro}).
\\

\subsection{$BV(\Omega)$  functions}
For $\Om$ an open set of $\R^n$, let $BV(\Om)$ be the space of functions 
$u \in L^1(\Om)$ 
 having as distributional derivative $Du$
a measure with finite total variation. For $u\in BV(\Om)$, we denote by $S_u$ the complement of the Lebesgue set of $u$. That is, $x \notin S_u$ if and only if 
$\lim_{r \to 0^+}\displaystyle \frac{1}{|B_r|} \int_{B_r(x)} |u(y)-z| \ dy =0$
for some $z \in \R$.
 We say that $x$ is an approximate jump point of $u$ 
if there exist  $\nu \in \Sn$ 
and distinct $a, b \in \R$
 such that
\[ \lim_{r \to 0} \frac{1}{|B_r^+(x,\nu)|} \int_{B_r^+(x,\nu)} |u(y)-a| \ dy =0 \quad \textrm{ and } \quad  \lim_{r \to 0} \frac{1}{|B_r^-(x,\nu)|} \int_{B_r^-(x,\xi)} |u(y)-b| \ dy =0,\]
where $B_r^\pm(x,\nu):= \{ y \in B_r(x) : \pm \sprod{y-x}{\nu}>0 \}.$ Up to a permutation of $a$ and $b$ and a change of sign of $\nu$, this characterizes the triplet $(a,b,\nu)$ which is then denoted by $(u^+,u^-, \nu_u)$. The set of approximated jump points is denoted by $J_u$. 
The following  theorem holds
\cite{AmbFusPa}.
\begin{theo}
 The set $S_u$ is countably $\hausmoins$-rectifiable and $\hausmoins(S_u\backslash J_u)=0$. Moreover $Du \restr J_u=(u^+-u^-) \nu_u \hausmoins \restr J_u$.
\end{theo}
We indicate by $Du= \nabla u \ dx \ + \ D^s u$ the Radon-Nikodym decomposition of $Du$. 
Setting $D^c  u:= D^s u \restr (\Om \backslash S_u)$ we get the decomposition
\[Du= \nabla u \ dx \ +\ (u^+-u^-)\nu_u \hausmoins \restr J_u\ +\  D^c u,\]
where $\restr$ denotes the  restriction. In particular, if $u=\pi\1_{E}\in BV(\Om,\{0,\pi\})$ then $Du= \pi \nu^E \hausmoins\restr \partial^*E$ where $\partial^*E$ is the reduced boundary of $E$ defined by
\[\partial^* E:= \left\{ x \in \textrm{Spt}( |D \1_E|) \,: \, \nu^E(x):=- \lim_{r \downarrow 0} \frac{ D \1_E (B_r(x))}{|D \1_E|(B_r(x))} \; \textrm{exists and } \; |\nu^E(x)|=1 \right\}\]
and $\nu^E$ is the outward measure theoretic normal to the set $E$ which is countably $\hausmoins$-rectifiable.
When $n=1$ we use the symbol $u'$ in place of $\nabla u$,
and $u(x^\pm)$ to indicate the right and left limits at $x$.

\subsection{Slicing method}\label{sub:slicing}
 In this section we recall the slicing method for functions with bounded variation \cite[Ch. 4]{braides} which will be used in the proof of the lower $\Gamma$-limit. 
Consider an open set $A \subset   \R^n$ and let $\nu \in \Sn$. We call $\Pi_\nu$ the hyperplane orthogonal to $\nu$ and $A_\nu$ the projection of $A$ on $\Pi_\nu$. We define the one dimensional slices of $A$, indexed by $x \in A_\nu$, as

\ben
	 A_{\nu x}: = \{ t \in \R \,;\, x+t\nu \in A\} \,.
\een

For every function $f$ in $\R^n$, we note $f_{\nu x}$ the restriction of $f$ to the slice $A_{\nu x}$, defined by $f_{\nu x}(t) := f(x+t\nu)$ .
Functions in $BV(\Om)$ can be characterised 
by one-dimensional slices (see \cite{braides}).
\begin{theo}\label{teo:slice}
Let $u \in BV(A)$.
Then for all $\nu \in \Sn$ we have 
$$
u_{\nu x} \in BV(A_{\nu x}) \qquad {\rm for}~ \hausmoins- {\rm a.e.}~ 
x \in A_\nu. 
$$
Moreover, for such points $x$, we have
\begin{equation}\label{egalprim}
 u'_{\nu x }(t)= \sprod{\nabla u(x+t \nu)}{\nu} \quad \textrm{for a.e. } t \in A_{\nu x},
\end{equation}
\begin{equation}\label{jumpslice}
 J_{u_{\nu x}}=\{ t \in \R : x+t \nu \in J_u\},
\end{equation}
and
\begin{equation}\label{formjumpslice}
 u_{\nu x} (t^\pm)=u^\pm(x+t \nu) \quad \textrm{or} \quad u_{\nu x} (t^\pm)=u^\mp(x+t \nu),
\end{equation}
according to whether $\sprod{\nu_u}{\nu}>0$ or $\sprod{\nu_u}{\nu}<0$.
Finally, for every Borel function $g: A \to \R$, 
\begin{equation}\label{integjumpslice}
 \int_{A_\nu} \sum_{t \in J_{u_{\nu x}}} g_{\nu x} (t ) 
\ d\hausmoins (x)=\int_{J_u} g~ |\sprod{\nu_u}{\nu}| \ d\hausmoins.
\end{equation}
Conversely if $u \in L^1(A)$ and if for all $\nu \in \{e_1,\dots, e_n\}$, where $(e_1, \dots, e_n)$ is a basis of $\R^n$,  and  almost
every $x \in A_\nu$ we have $u_{\nu x} \in BV(A_{\nu x})$ and
\[\int_{A_\nu} |Du_{\nu x}|(A_{\nu x})\ d\hausmoins(x) < +\infty,\]
then $u \in BV(A)$. 
\end{theo}

 \section{Derivation of the energy $\Feb$ from the coupled Gross-Pitaevskii functional} \label{Sec:derenergy}   

A two-components condensate is described by two functions $u_1$ and $u_2$, where $|u_1|^2$ and $|u_2|^2$ respectively represent the densities of the first and second component.
 The energy of the two-components condensate is given by a coupled Gross-Pitaevskii functional. When the intracomponent coupling strength of each component is equal to $1/\eps^2$, and when the intercomponent coupling strength is equal to $(1+\beta) /\eps^2$,
 the functional is given by   

\ben
	\mathcal{E}_\eps(u_1,u_2) :=  	 E_\eps(u_1)  + E_\eps(u_2)  + \frac{1+\beta}{2 \eps^2} \int_{\R^n}    |u_1|^2 |u_2|^2 \, dx \,, 
\een       

where $E_\eps$ is defined in (\ref{defGP1}). Assuming that the mass of each component is preserved, the functional $\mathcal{E}_\eps$ is minimised under the restrictions

\be \label{massu1u2}
	\int_{\R^n}   |u_1|^2 \, dx  = \alpha_1 \qquad \text{ and }  \qquad \int_{\R^n}   |u_2|^2 \, dx= \alpha_2 
\ee  

with $\alpha_1, \alpha_2 >0 $ and $\alpha_1 +\alpha_2 =\|\eta \|_2=1$. \\

Standard arguments used in the study of a single component condensate yield that the minimisers of $\mathcal{E}_\eps$ under the constraint (\ref{massu1u2}) are smooth
 positive functions, up the multiplication by constant terms of modulus $1$, with $L^\infty$ norm  uniformly bounded with respect to $\eps $ (see \cite{AfLivre,AftRoyo,IgnMil,KarSour}). 
 Notice  also that for a radial potential $V$, if $(u_1,u_2)$ is a minimiser, then for any rotation $R$ of the space, 
$(u_1\circ R,u_2\circ R)$ is also a minimiser. In the single component case, the Euler-Lagrange equations imply uniqueness of the minimiser from which one can infer its radial symmetry. For two components condensates, 
 this is not the case anymore.\\

The relation between $\mathcal{E}_\eps$ and $\Feb$ was established in \cite{AftRoyo}.  Using the nonlinear sigma model representation \cite{KaTsuUe,MaAf} and 
the Lassoued-Mironescu trick to decompose the energy of a rotating single condensate \cite{LaMi}, the authors introduced the change of variables 

\be \label{chvar}
	v := \frac{\sqrt{|u_1|^2+|u_2|^2}}{\eta_\eps} \qquad \text{ and } \qquad \frac{\p}2 := \text{Arg} \left(\frac {|u_1|+i|u_2| } {\sqrt{|u_1|^2+|u_2|^2 }} \right)  \, 
\ee 

for any pair $(u_1,u_2)$ such that $\mathcal{E}_\eps(u_1,u_2) <\infty$ and $|u_1|^2+|u_2|^2>0$. The equality 

\be \label{decompenergy}
	\mathcal{E}_\eps(u_1,u_2) = \Feb(v,\p) + E_\eps(\eta_\eps)
\ee 

then holds, and the mass constraints in (\ref{massu1u2}) rewrite as in (\ref{massvphi}). Let us point out that in \cite{AftRoyo},  only the case $n=2$ is considered  but the proof carries over verbatim to any space dimension. 
As seen from \eqref{decompenergy} and the expression  \eqref{defFeb} of $\Feb$, there are two main advantages of the formulation of the problem in terms of the functions $(v,\p)$. 
On the one hand, it naturally identifies the leading order term $E_\eps(\eta_\eps)$.  On the other hand, it clearly shows that the second order contribution $\Feb$ 
is a singular perturbation type functional.\\

Notice that since the minimisers $(u_1,u_2)$ are uniformly bounded and $\eta_\eps$ does not vanish, for every compact set $K$ of $\D$, there exists a constant $C(K)$ such that $0< v\le C(K)$ in $K$.
 Moreover, it is readily seen from the definition that $\p\in[0,\pi]$. We are thus naturally led to minimize $\Feb$ in the class
\[Y(\D):=\left\{ (v,\p) \, : \, \textrm{for every compact set $K\subset \D$, } 0< v\le C(K) \textrm{ in $K$ and } \p\in[0,\pi]\right\}\]   
under the mass constraints \eqref{massvphi}. For a subset $A$ of $\D$, we introduce the localised version of $\Feb$:
\[\Feb(v,\p;A):=\frac{1}{2}\int_{A} \eta_\eps^2 |\nabla v|^2 +  \frac1{2\eps^2} \eta_\eps^4 (1-v^2)^2 +  \frac14   \eta_\eps^2 v^2  \, |\nabla \p|^2 + \frac1{4 \eps^2} \beta  \eta_\eps^4 v^4 \,  \sin^2 \p  \, dx    \,,\]
and 
\[Y(A):=\left\{ (v,\p) \, : \, \textrm{for every compact set $K\subset \D$, } 0< v\le C(K) \textrm{ in $K\cap A$ and } \p\in[0,\pi]\right\}.\]

Notice that, for any $(v,\p)\in Y(\D)$,
 defining

\be \label{chvarinv}
	u_1:=\eta_\eps v \cos(\p/2) \qquad \text{ and } \qquad u_2:=\eta_\eps v \sin(\p/2)   
\ee 

relation (\ref{decompenergy}) holds and we have $|u_1|^2+|u_2|^2>0$.\\

In the following  we are going to make the following assumptions on $V$:

\begin{hyp}\label{hypo}
$V$ is such that $V(x)\to +\infty$ when $|x|\to +\infty$ and there exist $C, a ,b,c >0$ such that if $\rho$ is the Thomas-Fermi profile defined in \eqref{defrhointro},

 \ba  
	\|\eta_\eps\|_\infty &<& C  \label{etainfty} \\
 		\|\eta_\eps\|_{L^2(\R^n \setminus \D)} &\leq&  C \, \eps^a    \label{estoutbulk} \\ 
	    |\eta_\eps(x) - \sqrt{\rho}(x) | &\leq& C \, \eps^c \, \qquad \textrm{if } \dist(x,\partial \D)> C \eps^b    \label{localC0cvetaepsrho} 
\ea 
\end{hyp}

We remark that for the harmonic potential $V(x)=|x|^2$, it was proven in \cite{IgnMil} that these conditions hold true in dimension $n=2$. Moreover, it can be checked that their proof carries over almost verbatim to any space dimension. 
Recently, Karali and Sourdis \cite{KarSour}, obtained that if $n=2$, Hypothesis \ref{hypo} holds if $V$ satisfies:
\begin{itemize}
 \item[(i)] $V$ is nonnegative and $C^1$,
\item[(ii)] there exist $C>1$, $p\ge 2$ such that $\frac{1}{C}(1+ |x|^p)\le V(x)\le C(1+|x|^p)$,
\item[(iii)] $\D$ is a simply connected bounded domain containing the origin with smooth boundary and such that $\frac{\partial V}{\partial \nu}>0$ on $\partial \D$. 
\end{itemize}
Notice that in their paper, Karali and Sourdis  prove that $\|\eta_\eps-\sqrt{\rho}\|_{L^\infty(\R^2)}\le C \eps^{1/3}$ \cite[Rem. 4.4]{KarSour} which is  stronger than \eqref{localC0cvetaepsrho}.
 They also claim that their proof should extend to any space dimension (see \cite[Rem. 3.12]{KarSour}) 
and that the fact that $\D$ is simply connected is superfluous (see \cite[Rem. 1.1]{KarSour}).     

\section{The surface tension at finite $\beta>0$}\label{Sec:surftens}
In this section, for  $\beta>0$ fixed, we study the following variational problem:
\begin{equation}
 \sigbrb:=\inf\left\{ \Eund(v,\p) \,  : \,v\ge 0, \,  0\le \p \le \pi,   \, \lim_{-\infty} \p= 0 \textrm{ and } \, \lim_{+\infty} \p =\pi\right\},
\end{equation}
where 
\begin{equation}\label{energ1D}
 \Eund(v,\p):=\frac{1}{2} \int_{-\infty}^{+\infty}  v'^2 + \Wr(v)+\frac{1}{4} v^2 \p'^2 +\frac{\beta}{4} v^4  \sin^2 \p \, dt,
\end{equation}

with $\Wr(v):=\frac{1}{2} \left(1-v^2\right)^2$.
%

Let us point out that if $\Eund(v,\phi)$ is finite then $\lim_{x\to \pm \infty} v(x)=1$.\\

We start by evaluating the energy necessary to connect $v$ from a given value $m>0$ to $1$.
\begin{lem}\label{minMMrho}
Let $m\in [0,1]$ then
 \[\inf\left\{ \int_{0}^{+\infty} v'^2 + \Wr(v)\, dt \, : \, v(0)=m \right\}=\sqrt{2} \left(\frac{2}{3}-m+ \frac{m^3}{3}\right),\]
and the optimal profile is given by $v_m:=\ds\tanh \left(\sqrt{\frac{1}{2}} \, t + c_m \right)$ where $c_m:=\tanh^{-1}(m)$.
\end{lem}
\begin{proof}
 As in the usual Modica-Mortola problem,
\[
 \inf_{v(0)=m }\int_{0}^{+\infty} v'^2 + \Wr(v) \, dt  =\sqrt{2} \int_m^1 (1-t^2) dt
=\sqrt{2} \left(\frac{2}{3}-m+ \frac{m^3}{3}\right).
\]

\end{proof}

 We now prove that we can restrict ourselves to functions $v$ which stay away from zero.
\begin{prop}\label{away}
 For every $\beta>0$, there exists $m^*=m^*(\beta)>0$ such that 
\[ \sigbrb=\inf\left\{ \Eund(v,\p) \,  :\,   v\in[m^*,1], \, \, \lim_{-\infty} \p= 0  \textrm{ and }  \, \lim_{+\infty} \p =\pi \right\}.\]
 
\end{prop}

\begin{proof}
 First, let us notice that by truncation, we can reduce ourselves to minimise among functions $v\in [0,1]$. Up to translation we can also assume that $\inf_{\R} v=v(0)$. Let $m\ge 0$, then for every  function $v$ such that $\inf_{\R} v=v(0)= m$ and every admissible $\p$,
\begin{align*}
 \Eund(v,\phi)&\ge \frac{1}{2} \left[\inf_{v(0)=m} \int_{-\infty}^0 v'^2 + \Wr(v)\, dt \right]
+\frac{1}{2} \left[\inf_{v(0)=m} \int_{0}^{+\infty} v'^2 + \Wr(v) \, dt \right]\\
& \quad + \frac{1}{2} \int_{\R} \frac{1}{4} v^2 \p'^2 +\frac{\beta}{4} v^4  \sin^2 \p \, dt\\
& \ge \left[\inf_{v(0)=m}\int_{0}^{+\infty} v'^2 + \Wr(v) \,dt\right]+\frac{1}{4} \int_{\R}  \beta^{1/2} v^3 |\sin \p||\p'| \, dt\\
&\ge \sqrt{2} \left(\frac{2}{3}-m+ \frac{m^3}{3}\right)+\frac{ \beta^{1/2}m^3}{4}\int_{\R}   |\sin \p||\p'| \, dt\\
&=\sqrt{2} \left(\frac{2}{3}-m+ \frac{m^3}{3}\right)+\frac{ \beta^{1/2}m^3}{4} \int_0^\pi   |\sin x| \, dx\\
&=\sqrt{2} \left(\frac{2}{3}-m+ m^3\left(\frac{1}{3}+ \frac{ \beta^{1/2}}{2\sqrt{2}}\right)\right) \,.
\end{align*}
Now, for $m\ge 0$ and $T>0$, consider the test functions defined by

\[v_{m,T}:=\begin{cases}
                   v_m(-t-T) & t<-T\\
		  m & t\in[-T,T]\\
		  v_m(t-T)  & t\ge T
                  \end{cases}
\quad \textrm{ and } \quad
\p_T := \begin{cases}
                   0 & t<-T\\
		  \frac{\pi}{2T}(t+T) & t\in[-T,T]\\
		  \pi  & t\ge T,
                  \end{cases}
\]
 then 
\begin{multline} \label{a}
\Eund(v_{m,T},\p_T)=\sqrt{2} \left(\frac{2}{3}-m+ \frac{{m}^3}{3}\right)+\frac{T}{2} (1-{m}^2)^2+ \frac{{m}^2\pi^2}{16 T}+\frac{1}{4}\beta {m}^4\int_0^T \sin^2\lt(\frac{\pi}{2T}(t+T)\rt)dt\\
= \sqrt{2} \left(\frac{2}{3}-m+ \frac{{m}^3}{3}\right)+\frac{T}{2} (1-{m}^2)^2+ \frac{{m}^2\pi^2}{16 T}+\frac{\beta}{8} {m}^4 T.
\end{multline}
Optimizing in $T$ we find $T_m:=\frac{{m} \pi}{2\sqrt{2}((1-{m}^2)^2+\frac{\beta}{4} {m}^4)^{1/2}}$ and 
\begin{equation}\label{eqEm}
 \Eund(v_{m,T_m},\p_{T_m})=\sqrt{2} \left(\frac{2}{3}-m+ \frac{{m}^3}{3}\right)+ \frac{\sqrt{2}}{4} {m} \pi \lt((1-{m}^2)^2 +\frac{\beta}{4} {m}^4\rt)^{1/2}.
\end{equation}
Let now (see Figure \ref{fig1})
\[\Psi(m):= \left( \frac{m^3}{3}-m\right)+ \frac{1}{4} {m} \pi \lt((1-{m}^2)^2 +\frac{\beta}{4} {m}^4\rt)^{1/2}\]
so that $\Eund(v_{m,T_m},\p_{T_m})=\sqrt{2} \, (\Psi(m) + \frac{2}{3})$ and let 
\[\ov m:=\argmin_{m\in[0,1]} \Psi(m).\]
Let us first notice that  since $\Psi(0)=0$ and $\Psi'(0)=\frac{\pi}{4}-1<0$, the minimum of $ \Psi$ is negative for every $\beta>0$. The function $\frac{m^3}{3}-m$ is decreasing in $[0,1]$ and $\Psi(\ov m)> -\frac{2}{3}$ hence there exists a 
unique $m^*(\beta)\in(0,1)$ such that $\frac{m^*(\beta)^3}{3}-m^*(\beta)=\Psi(\ov m)$.  We claim that
\[ \sigbrb=\inf\left\{ \Eund(v,\p) \,  :\,   \inf v\ge m^*(\beta), \, \, \lim_{-\infty} \p= 0  \textrm{ and }  \, \lim_{+\infty} \p =\pi \right\}.\]
Indeed, if $v$ is such that $\inf v\le m^*(\beta)$ and if $\p$ is any admissible function, then letting $m:=\inf v$, there holds
\begin{align*}\Eund(v_{\ov m,T_{\ov m}},\p_{T_{\ov m}})=&\sqrt{2} \, \Big(\Psi(\ov m) + \frac{2}{3}\Big)< \sqrt{2} \Big(\frac{m^3}{3}-m + \frac{2}{3}\Big)\\
\le& \sqrt{2} \left(\frac{2}{3}-m+ m^3\left(\frac{1}{3}+ \frac{ \beta^{1/2}}{2\sqrt{2}}\right)\right)\le \Eund(v,\phi)\end{align*}
so that we can construct a competitor with smaller energy than $(v,\p)$. 
\end{proof}

\begin{figure}[ht]
\centering
 \includegraphics[width=6cm]{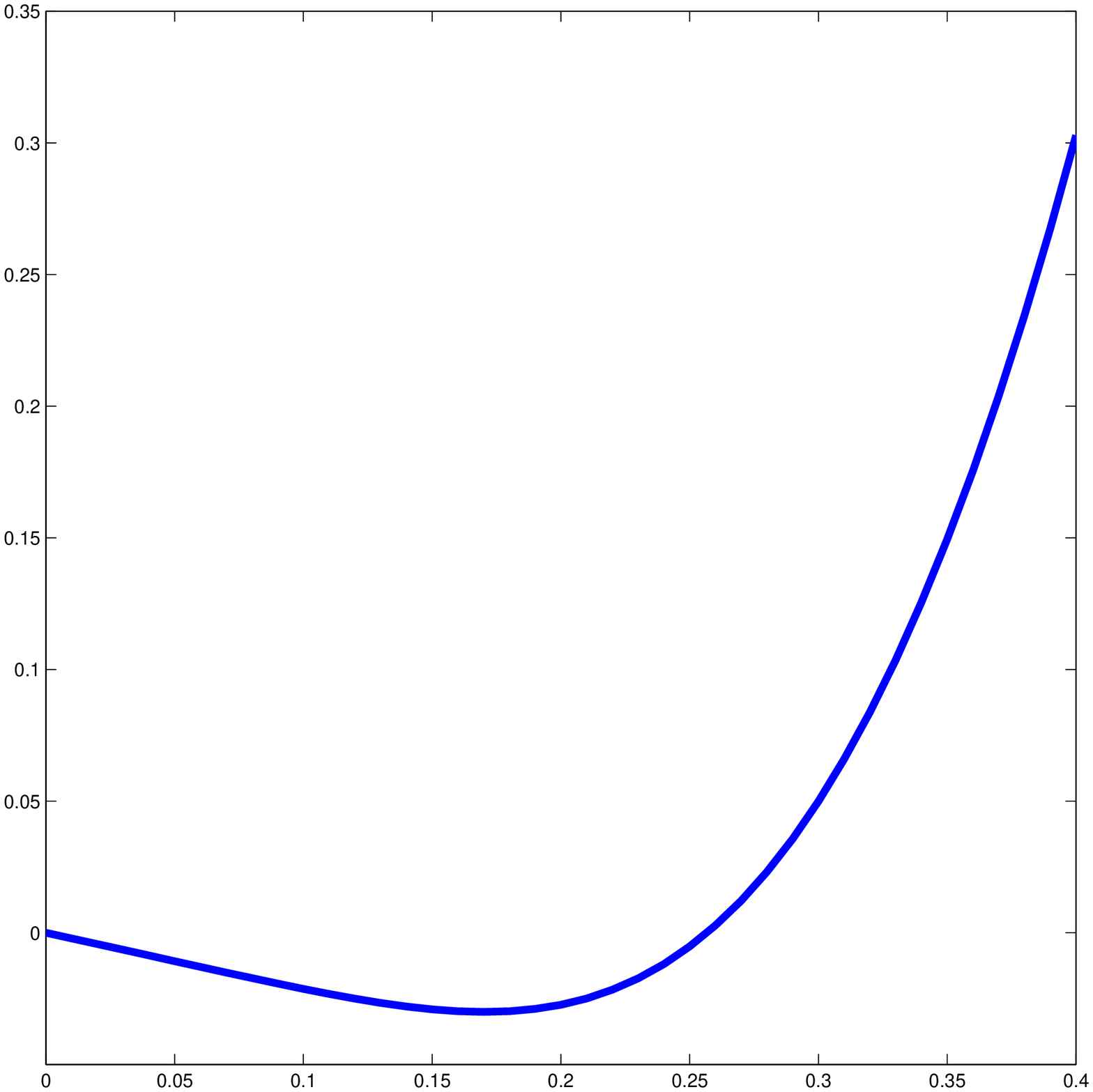}
\caption{The function $\Psi$}\label{fig1}
\end{figure}

 In the regime $\beta\to +\infty$, we can prove a more precise bound on $\inf v$. Notice that in the case $\beta=+\infty$, \cite{AftRoyo} proved that $\inf v =0$. 

\begin{prop}\label{decayinf}
There exist constants $B,C>0$ such that if $\beta\ge B$,  
\[ \sigbrb=\inf\left\{ \Eund(v,\p) \,  :\,   \frac{1}{C}\beta^{-1/4}\le \inf v\le C \beta^{-1/4}, \, \, \lim_{-\infty} \p= 0  \textrm{ and }  \, \lim_{+\infty} \p =\pi \right\}.\]

\end{prop}

\begin{proof}
 Let $\mathcal{M}:=\{m\in[0,1] :  m^3\left(\frac{1}{3}+ \frac{ \beta^{1/2}}{2\sqrt{2}}\right)-m>\Psi(\ov m)\}$ then arguing as in the previous proof, we obtain
\[ \sigbrb=\inf\left\{ \Eund(v,\p) \,  :\,   \inf v\notin \mathcal{M}, \, \, \lim_{-\infty} \p= 0  \textrm{ and }  \, \lim_{+\infty} \p =\pi \right\}.\]
The claim is thus proven provided we can show that for $\beta$ large enough, and for $m \in [0,1]$ such that $m\le \frac{1}{C}\beta^{-1/4}$ or $m\ge  C \beta^{-1/4}$ then $m\in \mathcal{M}$.
 We notice first that if $\beta$ is large then if $m\ge  C \beta^{-1/4}$, $m^3\left(\frac{1}{3}+ \frac{ \beta^{1/2}}{2\sqrt{2}}\right)-m>0>\Psi(\ov m)$ hence $m\in \mathcal{M}$.  
Taking $m=\tilde m \beta^{-1/4}$ with $0<\tilde m< \left(4\left[\frac{16-\pi^2}{\pi^2}\right]\right)^{1/4}$ so that  $\frac{1}{4} \pi \lt((1-{m}^2)^2 +\frac{\beta}{4} {m}^4\rt)^{1/2}<1$, we obtain $\Psi(\ov m)\le - \frac{1}{C}\beta^{-1/4}$
 and therefore, for $m\le \frac{1}{C} \beta^{-1/4}$, we have  $m^3\left(\frac{1}{3}+ \frac{ \beta^{1/2}}{2\sqrt{2}}\right)-m>\Psi(\ov m)$, that is $m\in \mathcal{M}$.
\end{proof}

We can now prove the existence of an optimal profile.
\begin{prop}\label{existbeta}
 For every  $\beta>0$ there exists a minimiser of $\sigbrb$. Moreover, it  is smooth and satisfies the Euler-Lagrange equations
\begin{eqnarray}
 -v''-(1-v^2)v+\frac{1}{4}v \p'^2+\frac{\beta}{2}  v^3 \sin^2\p&=&0  \label{ELv} \\
 \quad -(v^2\p')'+\beta v^4 \sin \p \cos \p&=&0 \label{ELphi} \,.
\end{eqnarray}
\end{prop}
\begin{proof}
 Let $(v_n,\p_n)$ be a minimising sequence. Up to translation, we can assume that $\p(0)=\frac{\pi}{2}$. Let us notice that up to truncating $v_n$, we can also 
assume that $v_n\in[0,1]$. Therefore, since $v'_n$ is uniformly bounded in $L^{2}(\R)$, up to extraction, the sequence $v_n$ converges 
locally uniformly to some continuous function $v$. Moreover, by lower semicontinuity,
\[\int_\R v'^2+\Wr(v) dt\le \varliminf_{n\to +\infty} \int_\R v_n'^2+\Wr(v_n) dt.\]
Since 
\[\int_\R( 1-v)^2 dt \le \int_\R (1-v^2)^2 dt \le C \qquad \textrm{and} \qquad \int_{\R} (1-v)'^2\le C,\]
the function $(1-v)$ is in $H^{1}(\R)$ and therefore $\lim_{\pm \infty} (1-v)=0$, i.e.   $\lim_{\pm \infty} v=1$. Thanks to Proposition \ref{away}, $\inf v_n\ge m^*$ from which 
we obtain that $\p'_n$ is bounded in $L^2(\R)$ and thus $\p_n$ also converges locally uniformly to some continuous function $\p$ with $\p(0)=\frac{\pi}{2}$. By lower semicontinuity, there holds
\[ \int_\R \frac{1}{4} v^2 \p'^2 +\frac{\beta}{4} v^4  \sin^2 \p \, dt\le \varliminf_{n\to +\infty} \int_\R \frac{1}{4} v_n^2 \p_n'^2 +\frac{\beta}{4} v_n^4  \sin^2 \p_n \, dt. \]
 
Since $\sin^2 \phi \in H^1(\R)$, the function $\sin^2 \phi$ converges to $0$  both at  plus and minus infinity so that $\phi$ has a limit at infinity which is either $0$ or $\pi$. Moreover, since $\phi(0)=\frac{\pi}{2}$ we see that $\phi$ cannot be constantly equal to $0$ or $\pi$ on $\R$. If $\lim_{x\to -\infty} v(x)=\lim_{x\to+\infty} v(x)$ then
assuming that
\[ \int_{-\infty}^0 v'^2+\Wr(v)+\frac{1}{4} v^2 \p'^2 +\frac{\beta }{4} v^4 \sin^2 \p \, dt\le \int_0^{+\infty} v'^2+\Wr(v)+\frac{1}{4} v^2 \p'^2 +\frac{\beta}{4} v^4  \sin^2 \p \, dt\]
and setting 
\[\tilde v(x):=\begin{cases}
                   v(x) & x<0\\
		  v(-x) & x\ge 0
                  \end{cases}
\quad \textrm{ and } \quad
\tilde \p(x) := \begin{cases}
                   \p(x) & x<0\\
		  \pi-\phi(-x) & x\ge 0,
                  \end{cases}
\]
we see that $\Eund(\tilde v,\tilde \phi)\le \Eund(v,\p)$ and up to symmetrising again, $\lim_{x\to -\infty} \tilde \p =0$ and $\lim_{x\to +\infty} \tilde \p =\pi$ so that $(\tilde v, \tilde \p)$ is a minimiser of $\Eund$. \\

From the integrated form of the Euler-Lagrange equations we see that $(v^2 \p')$ is in $H^1_{\loc}(\R)$ with derivative equal to  $\beta v^4 \sin \p \cos \p$ which is continuous.
 Hence, $v^2\p'\in C^1$ which implies (by continuity of $v$) that $\p'\in C^0$ and thus $\p\in C^1$. From this, we can use the first equation to infer higher regularity of $v$ and then a simple bootstrapping argument gives the smoothness of $(v,\p)$.
 
\end{proof}
\begin{rk}\rm
 Arguing as in \cite{AmbTort}, we could have obtained the existence of an optimal profile even without using the fact that $\inf v>0$.
\end{rk}

We can now study some qualitative properties of the minimisers of $\Eund$ at fixed $\beta>0$.
\begin{prop}\label{propertymin}
For every minimising pair  $(v,\p)$ of $\Eund$, the function $\p$ is increasing.  Moreover there exists a minimising pair $(v,\p)$ such
 that $\p(-t)=\pi-\p(t)$ and $v(-t)=v(t)$, $v$ is increasing on $\R^+$, $\p$ is convex on $\R^-$  and concave on $\R^+$. For every minimising function $v$, the minimiser of
\begin{equation}\label{uniqv}\min \left\{\int_\R \frac{1}{4} v^2 \p'^2 +\frac{\beta}{4} v^4  \sin^2 \p \, dt \, : \,\lim_{-\infty} \p= 0  \textrm{ and }  \, \lim_{+\infty} \p =\pi \right\}   \end{equation}
is unique and vice-versa, for every admissible $\p$, the minimiser of 
\begin{equation}\label{uniqphi}\min \left\{ \int_\R v'^2 + \Wr(v)+\frac{1}{4} v^2 \p'^2 +\frac{\beta}{4} v^4  \sin^2 \p \, dt \right\}\end{equation}
is unique. Finally, for every minimising pair $(v,\p)$, there is equipartition of the energy in the sense that
\begin{equation}\label{equiener}
 v'^2+\frac{1}{4} v^2 \p'^2  \,=\Wr(v)+\frac{1}{4} \beta v^4  \sin^2 \p.
\end{equation}

\end{prop}
\begin{proof}
 Let $(v,\p)$ be a minimising pair of $\Eund$ and let us prove that $\p$ is increasing. Let $t^-$ be the first point such that $\p(t)=\frac{\pi}{2}$ and  similarly, let $t^+$ be the last point such that $\p(x)=\frac{\pi}{2}$. If $t^-\neq t^+$ then assuming that 
\[\int_{-\infty}^{t^-}  v'^2 + \Wr(v)+\frac{1}{4} v^2 \p'^2 +\frac{\beta}{4} v^4  \sin^2 \p \, dt\ge \int_{t^+}^{+\infty}v'^2 + \Wr(v)+\frac{1}{4} v^2 \p'^2 +\frac{\beta}{4} v^4  \sin^2 \p \, dt,\] 
letting 
\[\tilde v(t):=\begin{cases}
                   v(t+t^+) & t\ge0\\
		  v(t^+-t) & t\le 0
                  \end{cases}
\quad \textrm{ and } \quad
\tilde \p(t) := \begin{cases}
                   \p(t+t^+) & t\ge 0\\
		  \pi-\phi(t^+-t) & t\le 0,
                  \end{cases}
\]
there holds $\Eund(\tilde v,\tilde \phi)<\Eund(v,\p)$ which gives a contradiction. From this, we see that $\p$  can take the value $\frac{\pi}{2}$ in only one point which up to translation can be assumed to be $0$. From this, it follows that $\phi>\frac{\pi}{2}$ in $\R^+$ hence from \eqref{ELphi}, we see that $v^2 \phi'$ is decreasing in $\R^+$. Since $\lim_{x\to +\infty} \phi(x)=\pi$ and $\phi(x)\le \pi$, there must be arbitrarily large $x$ such that $\phi'(x)\ge 0$ from which we infer that $\phi'$ is non-negative in $\R^+$. Similarly we can prove that $\phi'$ is also non-negative in $\R^-$.
   Let us  notice that the  symmetrisation made above,  constructed  a minimising pair 
$(\tilde v,\tilde \p)$ which satisfies $\tilde \p(-t)=\pi-\tilde \p(t)$ and $\tilde v(-t)=\tilde v(t)$.  From now on, let us drop the tildes for the sake of clarity and assume that $(v,\p)$ is a symmetric minimising pair.
%

 Let us now prove that we can further modify $ v$, respectively  $\p$, on $\R^+$ and get an increasing,   respectively a concave, function on $\R^+$ while decreasing the energy.
For this, we use standard rearrangement techniques (see \cite{liebloss}). For a function $f$ vanishing at infinity, let us denote by $f^*$ its decreasing rearrangement (see \cite{liebloss}). Analogously, for a function $g$ with limit $\alpha$ at infinity let us denote by $g_*$ its
 increasing rearrangement i.e. $f_*:=\alpha-(\alpha-f)^*$. From \cite[Th. 3.4]{liebloss}, we see that for two nonnegative functions $f$ and $g$ such that $f$ vanishes at infinity and $g$ has a limit at infinity, there holds
\[\int_{\R^+} f^*g_* \, dt \le \int_{\R^+} f g \, dt.\]  
Consider now $v_*$ the increasing rearrangement of $v$ then $\Wr(v_*)=\Wr(v)^*$, $(v^2)_*=(v_*)^2$, $(v^4)_*=(v_*)^4$ and $\displaystyle \int_{\R^+} v_*'^2 dt\le \int_{\R^+} v'^2 dt$.  Let finally $\widetilde \p:= \ds\frac{\pi}{2} +\int_0^x (\p')^*(t) dt$ be the primitive of the decreasing rearrangement of $\phi'$. 
Notice that $\widetilde \p$ is increasing and concave and for $x\in \R^+$, there holds
\begin{align*}
 \widetilde \p(x)&=\frac{\pi}{2}+\int_0^x (\p')^*(t) dt= \frac{\pi}{2}+\int_{\R^+} (\p')^*(t) \car{[0,x]} (t)dt\\
&=\frac{\pi}{2}+\int_{\R^+} (\p')^*(t) \car{[0,x]}^* (t)dt \ge \frac{\pi}{2}+\int_{\R^+} \p'(t) \car{[0,x]} (t)dt\\
&\ge\p(x),
\end{align*}
from which $\sin^2(\widetilde \p(x))\le \sin^2(\p(x))$ and by symmetry the same inequality holds in $\R^-$. From this, we infer that $\ds\int_{\R} (v_*)^4 \sin^2 (\widetilde \p) dt \le \int_{\R} v^4 \sin^2  \p dt$ and 
$\ds\int_{\R} (v_*)^2 (\widetilde \p')^2 dt \le \int_{\R} v^2 ( \p')^2 dt $. Putting all this together, we find that
\[\Eund(v_*,\widetilde \p)\le \Eund(v,\p).\]
Let  $v$ be a fixed minimising  function and let us prove that the minimiser of \eqref{uniqv} is unique. For this we use an observation of \cite{CapMelOt} (see also \cite{CherMur}) and let $\psi:= \sin \p$. The functional takes then the form
\[\int_\R  \frac{1}{4} v^2 \frac{\psi'^2}{1-\psi^2} +\frac{\beta }{4} v^4 \psi^2 \, dt\]
which is a strictly convex functional in $\psi$. From this we deduce that $\sin \p$ is unique and since $(v,\p)$ is minimising $\Eund$, the function $\p$ is increasing from which we infer that $\p$ is also unique. \\
Similarly, if $\p$ is any admissible function, then using the celebrated Brenier trick in optimal transportation, we let $w:= v^2$ and notice that the functional  can now be written as
\[\int_\R \frac{w'^2}{w} +\frac{1}{2}(1-w)^2+\frac{1}{4} w \p'^2 +\frac{\beta}{4} w^2  \sin^2 \p \, dt \]
which is strictly convex in $w$. Hence, $w$ is unique from which it follows that $v$ is also unique.

Finally,  the equipartition of the energy \eqref{equiener} follows simply by differentiating for instance the right handside and then using \eqref{ELv} and \eqref{ELphi}.
\end{proof}
\begin{rk}\rm
 If $v$ is any admissible function we cannot in general infer  that a minimising $\p$ of \eqref{uniqv} is increasing. In this case, we can however still conclude that $\sin \p$ is unique. 
\end{rk}

\begin{rk}\rm
 The uniqueness of the minimising pairs $(v,\p)$ seems to be a difficult question. Let us notice that the functional
\[\int_\R \frac{w'^2}{w} +\frac{1}{2}(1-w)^2+\frac{1}{4} w \frac{\psi'^2}{1-\psi^2} +\frac{\beta}{4} w^2  \psi^2 \, dt\]
 is not convex in $(w,\psi)$. Moreover, due to the non monotonicity of $v$, the sliding technique  (see \cite{BereTeraWaWei}) seems to be difficult to use here.
 We also mention that using the change of variables in (\ref{chvarinv}) with $\eta_\eps$ replaced by $ \sqrt{\rho}$, the uniqueness of the minimising pair $(v,\p)$
 would be  equivalent to the uniqueness of minimising pairs of  
 
 \ben
	\frac12 \int_\R u_1'^2+u_2'^2+ \frac12 (u_1^2+u_2^2-\rho)^2 + \beta u_1^2u_2^2 
\een
 
with constraints

 \ben
	\lim_{+\infty}  u_1 = \lim_{-\infty}  u_2 = \rho  \quad \text{ and } \quad  \lim_{-\infty}  u_1 = \lim_{+\infty}  u_2 = 0 \,.
\een


\end{rk}

\section{$\Gamma$-convergence of $\Feb$ for $\beta>0$}\label{Sec:gammacv}

In this section we study the $\Gamma-$convergence of the functionals $\eps\Feb$ as $\eps\to 0$ and prove Theorem \ref{gammavctheo}.

\subsection{Lower bound and compactness}

We start by proving the compactness of sequences with bounded energy.

\begin{prop}  \textbf{(Compactness)}
	Let $(v_\eps,\p_\eps)\in Y(\D)$ be a sequence of functions  such that 

\be \label{boundenergy2d}
	\sup_{\eps >0} \eps \Feb(v_\eps,\p_\eps) < \infty \,.
\ee	
	
	 Then, as $\eps \to 0$, 
	 
\ben
	(v_\eps,\p_\eps) \to (v,\p) \quad \text{in}  \quad L^1_{loc}(\D) \times L^1_{loc}(\D) \,,
\een
where $v=1$ a.e.\ in $\D$ and $\p \in BV_{loc}(\D;\{0,\pi\})$. Moreover, if $(v_\eps,\p_\eps)$ satisfy the mass constraint (\ref{massvphi}), then $\phi$ satisfies (\ref{mcfinal}). 
\end{prop}

\begin{proof} Let $K$ be an open set relatively compact  in $\D$. From (\ref{localC0cvetaepsrho}), there is $c=c(K)>0$ such that for $\eps$ small enough $\eta_\eps   > c  >0$ in $K$, so	 

\ban
	\int_{K}   |1-v_\eps |^2 + \beta v_\eps^4 \sin^2 \p_\eps     \leq \frac8{c^4} \,  \eps^2 \Feb(v_\eps,\p_\eps) =o_{\eps \to 0}(1)  \,.  
 \ean		 

Hence,  $v_\eps \to 1 $ in $L^2(K)$ and $\sin^2(\p_\eps) \to 0$ a.e. in $K$. We also observe that

\ben
	 \eps \Feb(v_\eps,\p_\eps) \geq  \frac{ c ^3}{4}   \int_{K}   |\nabla v_\eps|\, |1-v_\eps^2 |  +   v_\eps^3 \, |\nabla \p_\eps|  \, \sin\p_\eps    \geq c'(K)  \int_{K} |\nabla \psi(v_\eps,\p_\eps)| \,,
\een	 

where $\psi(s,t):=g(t)v^3(4/3-v)$ with $g(t):=\int_0^t \sin z \,dz=1- \cos t $. The functions $\psi(v_\eps,\p_\eps)$ are uniformly bounded in $BV(K)$,
 so  $\psi(v_\eps,\p_\eps) \to \psi_0$ in $L^1(K)$. We derive that  $g(\p_\eps) \to 3 \psi_0   $, which implies that $ \p_\eps \to \p = g^{-1} (3 \psi_0)\in L^1(K; \{0,\pi\})$, since $g$ is monotone and $\sin^2(\p_\eps) \to 0$.
 Then, since $\psi_0\in BV(K;\{0,2/3\})$, we obtain that $\phi\in BV(K; \{0,\pi\})$.  \\
 
 Finally, if $(v_\eps,\p_\eps)$ satisfy \eqref{massvphi}, then since $\int_\D \rho dx=1$, there is  $ r_K >0$ going to zero as $\text{dist}(K,\partial\D) \to 0$, 
such that $\int_{\D \setminus K} \rho dx =r_K$. Also, from  (\ref{localC0cvetaepsrho}) we have $\int_K \eta_\eps^2 v_\eps^2 \, dx \ - \ \int_K \rho dx= r_{\eps,K}=o_{\eps \to 0}(1)$. Combining these and $\int_{\R^n} \eta_\eps^2 v_\eps^2 dx=1$, 
we obtain $\lt|\int_{\R^n\backslash K} \eta_\eps^2 v_\eps^2 \cos \phi_\eps dx \rt|\le \int_{\R^n\setminus K} \eta_\eps^2 v_\eps^2  dx= r_K + r_{\eps,K} $, which yields   
\[ \lt|\int_K \rho \cos \phi dx -(\alpha_1-\alpha_2) \rt| =\lim_{\eps\to 0} \lt|\int_K \eta_\eps^2 v_\eps^2 \cos \phi_\eps dx-(\alpha_1-\alpha_2) \rt|= \lim_{\eps\to 0} \lt|\int_{\R^n\backslash K} \eta_\eps^2 v_\eps^2 \cos \phi_\eps dx \rt|   \leq  r_K\] 
and finishes the proof.  
   \end{proof}
 
In order to apply the slicing method we need to define the one dimensional restriction of the energy. For this we recall that for $A$ an open set of $\D$, $x\in A$ and $\nu\in \Sn$, we set $A_{\nu x} := \{ t \in \R \,;\, x+t\nu \in A\}$.
For $(v,\p) \in Y( A_{\nu x})$, we  define the one dimensional  energy 

\ban 
	\Feb(v,\p\,;A_{\nu x}) &:=&  \frac{1}{2}\int_{A_{\nu x}} \eta_{\nu x,\eps}^2 v'^2 +  \frac1{2\eps^2} \eta_{\nu x,\eps}^4 (1-v^2)^2 +  \frac14   \eta_{\nu x,\eps}\eps^2 v^2  \,  \p'^2 + \frac1{4 \eps^2} \beta  \eta_{\nu x,\eps}^4 v^4 \,  \sin^2 \p \, dt \,.  
\ean
We also define the limiting one dimensional energy as
\ban
\Fb(\p;\Anx):= \int_{\Anx} \frac{\sigma_{\nu x, \beta}}{\pi} |\p'|.
\ean
\begin{prop} \label{1dGammalimit} \textbf{(1d  $\Gamma-\liminf$)}
	Let $x \in \D$, $\nu \in\Sn$  and $\p \in BV_{loc}(\D_{\nu x};\{0,\pi\})$. For any sequence $(v_\eps,\p_\eps)  : \R_{\nu x}   \to (0,1] \times (0,\pi)$ converging as $\eps \to 0$ to $ (1,\p) $ in $L^1_{loc}(\D_{\nu x}) \times L^1_{loc}(\D_{\nu x}) $,  
	
\be \label{liminfF1dx}
	\liminf_{\eps \to 0}	\eps \Feb(v_\eps,\p_\eps;\D_{\nu x}) \geq \F_\beta(\p;\D_{\nu x})  \,.
\ee
\end{prop}

\begin{proof} Let $B$ be any open, relatively compact subset of $\D_{\nu x}$.  Let $ t_0 \in  B \cap J_\p$ and $\delta_0>0$ be such that  $  (t_0-\delta,t_0+\delta)$ is contained in $B$. We can choose $t^\pm \in  (t_0-\delta,t_0+\delta)  $ such that 

\ben
	t^- < t_0 < t^+  \,,\quad \p (t^+) \neq \p (t^-) \,,   \quad  \p_\eps(t^\pm) \to \p(t^\pm)  \in \{0,\pi \}    \quad \text{and} \quad v_\eps(t^\pm) \to 1\, 
\een	  

as $\eps \to 0$. Estimate (\ref{localC0cvetaepsrho}) and  $\overline B \subset   \D_{\nu x} $ yield  

\ban
	\eps \Feb(v_\eps,\p_\eps; (t^+,t^-) )&=&    \frac{1}{2} \int_{t^-}^{t^+}  \eps \eta_{\nu x,\eps}^2  v'^2_\eps+  \frac1{2\eps }   \eta_{\nu x,\eps}^4 (1-  v_\eps ^2)^2 +  \frac14   \eps \eta_{\nu x, \eps}^2   v_\eps^2  \,     \p'^2_\eps+ \frac1{4\eps } \beta   \eta_{\nu x,\eps}^4   v_\eps ^4 \,  \sin^2   \p_\eps dt \\
 							&\geq&   \rho_{\nu x}(t_0) \frac{1}{2}    \int_{t^-}^{t^+}  \eps     v'^2_\eps +  \frac1{2 \eps}  \rho_{\nu x}(t_0)  (1-  v_\eps ^2)^2 +  \frac14  \eps    v_\eps^2  \,    \p'^2_\eps + \frac1{4 \eps} \beta   \rho_{\nu x}(t_0)   v_\eps ^4 \,  \sin^2   \p_\eps dt  \\ &&    \, - \, c' \,   \delta  + o_{\eps \to 0}(1)  \,.
 \ean

for some $c'=c'(B)>0$.  We define $T^\pm_\eps:= (t^\pm - t_0) \frac{\sqrt{\rho_{\nu x}(t_0)}}{2\eps}$ and $\tilde f  (t) :=   f \lt(\frac{\eps}{\sqrt{\rho_{\nu x}(t_0)}} t + \tilde t_0\rt) $ for $f = v_\eps, \p_\eps, \eta_{\nu x,\eps} \text{ or } \rho_{\nu x}  $.   A change of variables yields

\ban
	\eps \int_{t^-}^{t^+}        v'^2_\eps +  \frac14   v_\eps^2  \, \p'^2_\eps dt  &=&   \sqrt{\rho_{\nu x}(t_0)}  \int_{T^-_\eps}^{T^+_\eps} \tilde v'^2_\eps +  \frac14 \tilde   v_\eps^2  \, \tilde  \p'^2_\eps dt   \\
	    \frac{\rho_{\nu x}(t_0)}{  \eps}    \int_{t^-}^{t^+}  \frac12     (1-  v_\eps ^2)^2 +  \frac1{4  } \beta     v_\eps ^4 \,  \sin^2   \p_\eps   dt &=&    \sqrt{\rho_{\nu x}(t_0)}   \int_{T^-_\eps}^{T^+_\eps} \frac12   (1-   \tilde v_\eps ^2)^2+    \frac1{4 \eps} \beta    \tilde v_\eps ^4 \,  \sin^2  \tilde \p_\eps dt 
\ean

and thus

\ben
	 \Feb(v_\eps,\p_\eps; (t^+,t^-) )   \geq     \rho_{\nu x}(t_0)^{3/2} \, \mathcal G_{\beta} ( \tilde v_\eps, \tilde \p_\eps ; (T^-_\eps,T^+_\eps))   \, - \, c' \,   \delta  + o_{\eps \to 0}(1)   \,,
 \een	 
where for an interval $I$ and a pair $(v,\p)$,
$\mathcal G_{\beta} (v, \p ; I)$ is the localized version of $\mathcal G_{\beta}$ defined in (\ref{energ1D}).

Define now 
 
\ben
	\hat v_\eps (t) := \left\{ \begin{array}{ccl}
			\tilde v_\eps (t)&  \text{if} & t \in  (T^-_\eps,T^+_\eps)   \\
			\text{ linear joint} & \text{if}  & t  \in (T^+_\eps    ,T^+_\eps  +\delta )  \cup  (T^-_\eps -\delta    ,T^-_\eps   ) \\
			1 & \text{if}  & t \in \R \setminus   (T^-_\eps  -\delta   ,T^+_\eps  +\delta )  	 
      		 	\end{array}  \right. 
\een
and 
\ben
	\hat \p_\eps (t) := \left\{ \begin{array}{ccl}
			\tilde \p_\eps (t) & \text{if}  &  t \in  (T^-_\eps,T^+_\eps)    \\
			\text{ linear joint} & \text{if}  & t  \in (T^+_\eps    ,T^+_\eps  +\delta )  \cup  (T^-_\eps -\delta    ,T^-_\eps   ) \\
			\p (t^-) & \text{if}  & t \in(-\infty,   T^-_\eps - \delta) \\
			\p (t^+)  & \text{if}  & t \in (T^+_\eps + \delta, +\infty)
       		 	\end{array}  \right.   \,.
\een

We have that $(\hat v_\eps,\hat \p_\eps)$ is admissible for $ \overline \sigma_{\beta}$ so

\ben
	\mathcal G_{\beta} ( \tilde v_\eps, \tilde \p_\eps ; (T^-_\eps,T^+_\eps)) \geq \bar{\sigma}_{ \beta} + o_{\eps\to 0}(1)    \,.
\een

 Hence,
  
\be \label{est00delta}
	 \eps \Feb(v_\eps,\p_\eps; (t^+,t^-) ) \geq  \sigma_{\nu x,\beta} (t_0)    + o_{\eps\to 0}(1)   - \, c'  \,   \delta    \,. 
 \ee

Since $\p \in BV_{loc}(\D_{\nu x}, \{0,\pi\})$ we have  $      B \cap J_\p = \{ t_0 , \dots , t_N \}$ for some $N \in \mathbb{N}$. Consider $\delta_0>0$ such that  for $\delta \in (0,\delta_0)$, the intervals  $I_\delta  = (t_0-\delta,t_0+\delta)$ are disjoint and contained in $B$. Reasoning as before and since (\ref{est00delta}) holds for every $\delta \in (0,\delta_0)$, we obtain

 \ben
	 \eps \Feb(v_\eps,\p_\eps; \D_{\nu x} ) \geq \sum_{i=0}^N \sigma_{\nu x,\beta} (t_i)   + o_{\eps\to 0}(1)     \,. 
 \een	 
 
Thus,

 \ben
 	\liminf_{\eps\to 0} \eps \Feb(v_\eps,\p_\eps; \D_{\nu x} ) \geq   \sum_{t \in B \cap J_\p }  \sigma_{\nu x,\beta} (t)    = \int_{B}   \frac{\sigma_{\nu x,\beta}}{\pi} |\p'|   =  \F_\beta(v,\p;B)  \,.  
\een

This yields  (\ref{liminfF1dx}) since the choice of $B$ was arbitrary.

 \end{proof}

We can now prove the $\Gamma-$ liminf. For any $\p \in BV_{loc}(\D)$, we define the localised  lower $\Gamma$-limit of $\eps \Feb $ as the set function defined in $ \mathcal{A}(\D)$ by

 \ben
	F' (\p;A) := \inf \left\{ \liminf_{\eps\to 0} \eps \F_\eps(  v_{\eps },   \p_{\eps}\, ;A )  \,;\,  (  v_{\eps },   \p_{\eps}) \to (1,\p) \text{ in } L^1_{loc}(\D) \times L^1_{loc}(\D)   \right\} \,,
\een

and we write  $	F' (\p) :=		F' (\p;\D)$.
\begin{prop} \textbf{($\Gamma-$liminf)}
For any $\p \in BV_{loc}(\D ;\{0,\pi\})$, 
	
\be \label{liminfF}
	F' (\p) \geq \F_\beta(\p )  \,.
\ee
\end{prop}

\begin{proof}  Consider any fixed open set  $A $ relatively compact in $\D$,  $\nu \in \Sn$ and $\p\in BV(A,\{0,\pi\})$. Let then   $(v_\eps,\p_\eps) $ be such that $v_\eps\to 1$ and $\p_\eps\to \p$ in $L^1(A)$ and such that 
\[\varliminf_{\eps\to 0} \eps \Feb(v_\eps,\p_\eps; A ) =F' (\p;A).\]
 We may assume that $F'(\p,A)<\infty$, so that \eqref{boundenergy2d} is satisfied. From Fubini's Theorem, there holds

\ben 
	 \eps\F_{\eps,\beta} (v_\eps,\p_\eps\,;A) \geq \int_{A_\nu}  \eps\F_{\eps,\beta} (v_{\eps,\nu x},\p_{\eps,\nu x}\,;A_{\nu x}) \, d\HH \,,
\een

with $(v_{\eps,\nu x},\p_{\eps,\nu x}) \to (1,\p_{\nu x})$  for a.e.\ $x \in A_\nu$. Then,  Fatou's lemma, Fubini's formula, \eqref{integjumpslice} and Proposition \ref{1dGammalimit} yield

\ben 
	\varliminf_{\eps \to 0}	\eps \Feb(v_\eps,\p_\eps;A) \geq  \int_{A_{\nu }}  d\HH(x)  \int_{A_{\nu x}}   \frac{\sigma_{\nu x,\beta}}{\pi} |\p'_{\nu x}|  =  \int_{ A \cap J_\p}  \sigma_\beta(x)   |\langle \nu_\p, \nu  \rangle|  d\HH    \,.
\een	

 Notice that the last equality holds because $\p$ is the characteristic function of a set with finite perimeter  in $A$. Hence,

\ben
 	F' (\p;A) \geq    \int_{ A \cap J_\p}  \sigma_\beta(x)   |\langle \nu_\p, \nu  \rangle|  d\HH  \,.
\een

Since all the functions $F_\eps$ are local,  $F' (\p;\cdot)$ is super-additive on open sets with disjoint compact closures.  We may apply  \cite[Prop. 1.16]{braides} with $\Omega=\D$, $\lambda =  \sigma_\beta(x) \,   \HH \restr \partial^*\{ \p = \pi\} $ (where we recall that $\partial^*E$ denotes the reduced boundary of $E$) and $\psi_i= |\langle \nu_\p, \nu_i \rangle| $,  where $\{\nu_i \}$ is a dense family in $\Sn$. Remarking that $  \sup_i  |\langle \nu_\p, \nu_i \rangle|  =1 $, we obtain 

\ben
	F' (\p;A) \geq \int_{A\cap J_\p}  \sigma_\beta(x)  d\HH =  \F_\beta(\p;A) \,, 
\een

which yields (\ref{liminfF}).  
  
\end{proof}


\subsection{$\Gamma-$limsup}
In this section we construct a recovery sequence and prove the $\Gamma-$limsup.  Using the following lemma, we may restrict our selves to prove the inequality for the $\Gamma-$limsup for functions in

 \ben
 	X := \left\{  \p = \pi \1_A    \,;\, A \text{ relatively compact open set of $\D$ of class } C^{\infty}   \right\}\,.
 \een

\begin{lem} \label{approxA}
	Let  $\p= \pi \mathbf{1}_A \in BV_{\text{loc}}(\D)$. There exists a sequence $\{ \p_k=\pi \1_{A_k} \}_{k \in \mathbb{N}}$  in  $X$ such that:
	\begin{description}
	\item[(i)] $\lim_{k \to \infty} \mathcal L^n((A_k \cap\D)\Delta A)=0$,
	\item[(ii)]  $\limsup_{k \to \infty}  \F_\beta(\p_k ) \leq  \F_\beta(\p )$,
	\item[(iii)]   $\int_{A_k } \rho dx = \int_{A} \rho dx$  \,.
\end{description}
\end{lem}
 
The proof of Lemma \ref{approxA} uses the continuity of $\sigb$ with respect to $x$ and follows closely the proof of \cite[Prop. 4.1]{Bou}, therefore we omit it here.  \\

We first construct in Proposition \ref{Glimsup} a recovery sequence for functions in $ X$. We then explain in Lemma \ref{mclemma} how to take into account the mass constraint  (\ref{massvphi}).

\begin{prop}[{\bf $\Gamma$-limsup}] \label{Glimsup}
	Let $\beta>0$ and $\p = \pi \1_A  \in X$, then there exists a sequence of functions $(v_\eps,\p_\eps) \in Y(\D)  $ such that

	\be \label{cvpropglsup}
		(v_\eps,\p_\eps)  \to (1,\p) \quad  \text{ in }  \quad  L^1_{loc}(\D) \times L^1_{loc}(\D)  
	\ee
	
and	
	\be \label{gammalimsup}
		 \varlimsup_{\eps \to 0}  \eps \Feb(v_\eps,\p_\eps) \leq   \F_\beta(\p) \,.
	\ee 
\end{prop}

\begin{proof}  Define the signed distance to $\partial A$ by $d(x) := \text{dist}(x,A) -  \text{dist}(x,\R^2 \setminus A) $. For sufficiently small $t >0$, 
the projection $\Pi$ on $\partial A$ is well defined in the set $\{ x \in D \,;\, |d(x)| < t    \} $ and $d$ is a Lipschitz function therein with $|\nabla d |=1$ 
a.e.\ .  Define also

 \ben
	f(\eta,v,\p,p,q) :=\frac{1}{2}\left(   \eta^2 |p|^2 +  \frac12 \eta^4 (1-v^2)^2 +    \frac 14   \eta^2 v^2  \, |q|^2 + \frac14 \beta  \eta^4 v^4 \,  \sin^2 \p\right) \,.	
\een

 Let $(v,\p)$ be a minimiser of $ \bar{\sigma}_{\beta}$ and for $x \in \D$, let $v_x(t):=v(\rho(x)^{1/2}t)$, $\p_x(t):=\p(\rho(x)^{1/2} t)$ and  for $\ell >0$ let

\ben
	v_{x,\ell}   := \  (1+\ell)v_x \wedge 1  \quad \text{and} \quad \p_{x,\ell} :=0 \vee \Big( \big( (1+2\ell)\p_x -\ell \big) \wedge 1 \Big) \,. 
\een 
 
Notice that $(v_{x,\ell} ,\p_{x,\ell} )$ converges pointwise to $(v_x,\p_x )$ as $\ell \to 0$, and that there exists $C>0$ such that for every $\ell \in (0,1)$ and $x \in \D$, 

\be \label{todct}
	f (\sqrt  \rho(x) ,v_{x,\ell}, \p_{x,\ell}, v'_{x,\ell} ,  \p'_{x,\ell}) \leq  C f (\sqrt  \rho(x) ,v_x, \p_x, v'_x,  \p'_x)\,.
\ee

 Therefore, thanks to the dominated convergence Theorem, for every $\delta>0$ and every $x\in \D$, there exists $\ell_x$ such that for $\ell\le \ell_x$,
\begin{equation}\label{estimdelta}
 \int_{\R} f (\sqrt  \rho(x) ,v_{x,\ell}, \p_{x,\ell}, v'_{x,\ell} ,  \p'_{x,\ell})  dt\leq  \sigb(x) +\frac{\delta}{2}.
\end{equation}

 Fix from now on such a $\delta>0$. Thanks to the compactness of $\partial A$ and the continuity of $\sigb$, there is a finite family $\left\{ \Sigma_i \right\}_{i \in I} $ of open disjoint subsets
 of $\partial A $ such that  $\HH \left(   \partial A \setminus \cup_{i \in I} \Sigma_i  \right) =0$ and 

\be \label{approunif}
	\sigb(x_i) \leq \sigb(x)+\frac{\delta}{2}  \quad \text{in} \quad  \Sigma_i  
\ee 

for every $i \in I$. Let then $\ell:=(\min_{i\in I} \ell_{x_i})\wedge \delta$ and define $	 \Sigma_i^\delta := \left\{ x  \in \Sigma_i \,;\, \text{dist}(x,\partial \Sigma_i) > \ell   \right\}  $  so that

\be \label{cvsiglali}
	\HH (\Sigma_i  \setminus \Sigma_i^\delta  ) = o_{\delta \to 0}(1) \,. 
\ee  

For $\eps,T>0$ define

\ban
	W_\eps  &:=&  \{ x \in \R^n \,;\, |d(x)| < \eps T  \}\\
	B_i &:=& \{ x \in W_\eps \,;\,    \Pi(x) \in \Sigma_i^\delta  \} \\
	 C_i &:=& \{ x \in W_\eps \,;\,   \Pi(x) \in \Sigma_i \setminus \Sigma_i^\delta  \} \,.
\ean

Notice that for every given $T$, for $\eps$ small enough $W_\eps$ is contained in some fixed compact set of $\D$ containing $A$. Consider a family $ \{   \theta_i  \}_{i\in I} $ of smooth functions such that

\ben
	\sum_{i\in I}   \theta_i = 1 \quad \text{on } \partial A \qquad \text{ and } \qquad    \theta_i = 1 \quad \text{in }    \Sigma_i^\delta   \quad     \forall \, i \in I  \,,
\een

and  define 

\be \label{defrecseq}
	( v_\eps, \p_\eps ) (x)  = \left\{ \begin{array}{cl}
	 \sum_{i\in I}    \theta_ i (\Pi(x)) \left( v_{x_i,\ell}\left( \frac{d(x)}\eps \right), \p_{x_i,\ell}\left( \frac{d(x)}\eps \right)  \right) & \text{ if } |d(x)| \leq \eps T_\delta   \\
	(1,\pi) & \text{ if }  d(x) \geq \eps T_\delta   \\
	(1,0) & \text{ if }  d(x) \leq -\eps T_\delta   \\
	\end{array}\right. \,;
\ee 

where $T_\delta$ is big enough so that $ v_{x_i,\ell}=1 $ in $\R \setminus [-T_\delta,T_\delta]$,   $\p_{x_i,\ell}=0 $ in $ (-\infty,T_\delta]$ and $\p_{x_i,\ell}=\pi $ in $[T_\delta,+\infty)$ for every $  i \in I $. \\

The functions $( v_\eps, \p_\eps ) $ are Lipschitz continuous and converge to $(1,\p   )$ in $L^1_{loc}(\D) \times L^1_{loc}(\D) $. Defining

\ben
	\xi_\eps :=  \frac 1\eps   f(\eta_\eps, v_\eps ,   \p_\eps, \eps^2  \nabla v_\eps,  \eps^2  \nabla \p_\eps) \,, 
\een

there exists $C>0$ such that 

\be \label{derphi}
	|\xi_\eps | \leq  C\eps^{-1}  \,,
\ee
 
and since $|\nabla d |=1$ in $W_\eps$,  

\be \label{phiinbi}
	\xi_\eps(x)  =    |\nabla {d}/{\eps} |  f(\eta_\eps, v_{x_i,\ell} \circ  {d}/{\eps}  ,   \p_{x_i,\ell} \circ  {d}/{\eps},  v'_{x_i,\ell} \circ  {d}/{\eps},  \p'_{x_i,\ell} \circ  {d}/{\eps})  
 \ee 

holds in $B_i$ for all $i \in I$. \\

Using (\ref{cvsiglali}) and (\ref{derphi}) we compute

\ba 
	\eps \Feb(v_\eps,\p_\eps)  &=& \sum_{i \in I}   \int_{   B_i}  \xi_\eps(x) \, dx  + \sum_{i \in I}   \int_{  C_i}   \xi_\eps(x) \, dx  	\nonumber \\
						&\leq&  \sum_{i \in I}       \int_{   B_i}  \xi_\eps(x) \, dx  + \frac C\eps    \mathcal L^n \left(  \bigcup_{i\in I}  C_i \right)  \label{est1}\\
						&=&  \sum_{i \in I}    \int_{  B_i}  \xi_\eps(x) \, dx  +   r^1_{\delta} \,  \nonumber
\ea  
 
 where $r^1_{\delta} =  o_{\delta \to 0}(1) $. Using (\ref{phiinbi})  and the coarea formula  \cite[Prop. 2.4]{Bou} applied to  $u=d/\eps$, we obtain

\ben
	 \int_{  B_i}  \xi_\eps(x) \, dx  =      \int_{-T_\delta}^{T_\delta} \int_{   \{ |d| = \eps t \} \cap B_i  }  f(\eta_\eps(x), v_{x_i,\ell} (t)  ,   \p_{x_i,\ell} (t) , v'_{x_i,\ell} (t),  \p'_{x_i,\ell}(t)) \, d\mathcal{H}^{n-1}(x)  \, dt \,.
 \een
  
Since $B_i \subset \subset \D$, estimate (\ref{localC0cvetaepsrho}) gives

\ban
	 \int_{   B_i}  \xi_\eps(x) \, dx   =      \int_{-T_\delta}^{T_\delta} \int_{   \{ |d| = \eps t \} \cap  B_i     }  f(\sqrt{\rho}(x), v_{x_i,\ell} (t)  ,   \p_{x_i,\ell} (t) , v'_{x_i,\ell} (t),  \p'_{x_i,\ell}(t))   \, d\mathcal{H}^{n-1}(x)  \, dt 
  + r^2_{\delta,\eps}     
  \ean

 where $r^2_{\delta,\eps} = o_{\eps \to 0}(1)$.\\
%
%

Hence, using Fubini's Theorem, \eqref{estimdelta} and \eqref{approunif} we find

\begin{align} 
	\varlimsup_{\eps \to 0}   \int_{    B_i}   \xi_\eps(x) \, dx    &\leq    \int_{     \Sigma_i^\delta} \int_{-T_\delta}^{T_\delta}   f(\sqrt \rho(x_i) ,v_{x_i,\ell}, \p_{x_i ,\ell}, v'_{x_i ,\ell}, \p'_{x_i ,\ell})  \, dt \, d \HH(x) \nonumber \\
&\leq  \int_{     \Sigma_i^\delta} (\sigb(x)  + \delta)  d\HH(x)   \label{est2} \,.
 \end{align} 
 
Putting together (\ref{est1}) and  (\ref{est2}) we obtain

 \ban
	\lim_{\delta \to 0} \varlimsup_{\eps \to 0}    \eps \Feb(v_\eps,\p_\eps)  &\leq&   \int_{  \partial A    } \sigb  d \HH\,= \mathcal{F}_\beta(\p).
 \ean 
 Finally,  a diagonal argument yields (\ref{gammalimsup}).   
  
\end{proof} 

\begin{lem} \textbf{(mass constraint)} \label{mclemma}
Let $\beta>0$ and $\p = \pi \1_A  \in X$ satisfying (\ref{mcfinal}). Then, there exists a sequence of functions $(v_\eps,\p_\eps)$ satisfying  (\ref{massvphi}) for every $\eps>0$ for which (\ref{cvpropglsup}) and  (\ref{gammalimsup}) hold. 
 \end{lem}

\begin{proof} Notice first that since $\ds \int_A \rho \, dx=\alpha_2>0$ and $\ds\int_{A^c} \rho \, dx=\alpha_1>0$, there exist $x^+$ in $ A$ and $x^-$ in $\D \setminus \overline{A}$. With the notations of Proposition \ref{Glimsup}, consider $( v_\eps,\p_\eps)$  as in (\ref{defrecseq}) with $d$ given by the signed distance to $A_\eps := (A \cup B^+_\eps) \setminus B^-_\eps$, where $B^{\pm}_\eps := B(x^\pm,\delta_\eps^\pm)$ and

\be \label{defbeta}
	0 \le \delta_\eps^\pm  \le \eps^{\gamma/n} \quad \text{ with } \quad \gamma \in (0,1) \,. 
\ee 

Defining $\hat v_\eps = \| \eta_\eps   v_\eps \|^{-1}_2   v_\eps $, the first equality in (\ref{massvphi}) holds. Using   $ \| \eta_\eps  \|_2=1 $ we estimate

\be  \label{estnorme2}
	 \| \eta_\eps   v_\eps \|^{2}_2 = 1 + \int_{  W_\eps }  \eta_\eps^2(  v_\eps^2-1) \, dx = 1 +O(\eps) \,. 
\ee 

Hence, the sequence $(\hat v_\eps,\p_\eps)$ converges to $ (1,\p) $ in $L^1_{loc}(\D) \times L^1_{loc}(\D) $ and inequality (\ref{gammalimsup}) still holds.\\   

Using estimates (\ref{etainfty})-(\ref{localC0cvetaepsrho})  we get   

\ban
	 \int_{\R^n}  \eta_\eps^2  v_\eps^2 \cos \p_\eps\, dx &=&  \int_{\R^n}  \eta_\eps^2 (-\1_{A_\eps} + \1_{\R^n\setminus A_\eps}) \, dx+   O(\eps) \\  
	 &=&  \int_{\R^n}  \rho (-\1_{A}  + \1_{\D \setminus A}) \, dx +   2 \int_{\R^n} \eta_\eps^2  ( \1_{B^+_\eps}-\1_{B^-_\eps}  ) \, dx  +   O(\eps^\tau) \\ 
 \ean

 where $\tau := \min \{a,b,c  \}>0$. Let $c>0$ be such that $\eta_\eps^2> c$ in $B^-_\eps\cup B^+_\eps$  and fix  $\gamma \in (0,\tau)$.  
 For $\eps$ small enough, thanks to \eqref{defbeta} and \eqref{estnorme2} we obtain for $ ( \delta_\eps^+ ,  \delta_\eps^-) = (\eps^\gamma,0) $,
 \ben
	 \int_{\R^n}  \eta_\eps^2   \hat v_\eps^2 \cos \p_\eps \, dx \ge \alpha_1- \alpha_2 + 2c |B_1|\, \eps^\gamma   + O(\eps^\tau)   > \alpha_1- \alpha_2  
 \een
   and for $ ( \delta_\eps^+ ,  \delta_\eps^-) = (0,\eps^\gamma) $,
 \ben
	 \int_{\R^n}  \eta_\eps^2   \hat v_\eps^2 \cos \p_\eps \, dx\le \alpha_1- \alpha_2 - 2c|B_1| \, \eps^\gamma   +O(\eps^\tau)   < \alpha_1- \alpha_2.    
 \een
  We conclude by continuity that there exists  $ ( \delta_\eps^+ ,  \delta_\eps^-) \in [0;\eps^\gamma] \times   [0;\eps^\gamma]$ such that the second equality in (\ref{massvphi}) is satisfied.

\end{proof}

\section{Asymptotic analysis of the surface tension}\label{Sec:surfasympt}
In this section we study the asymptotic behavior of $\sigbrb$ when $\beta$ tends to zero or infinity.
\subsection{Vanishing $\beta$}
When $\beta$ goes to zero, we expect the two condensates not to segregate anymore. This can be seen as an interpretation of the following theorem which shows that in the limit $\beta\to 0$, the surface tension $\sigbrb$ vanishes.
\begin{theo}
 The functional $\Eund $ $\Gamma$-converges when $\beta\to 0$ to 
\[\mathcal{G}_{0}(v,\p):= \frac{1}{2}\int_{\R} v'^2 +\Wr(v) + \frac{1}{4} v^2 \p'^2  \, dt,\]
 which is defined on all the pairs of functions $(v,\p)$ with $\p\in [0,\pi]$ (but without conditions at infinity). As a consequence,
\[\lim_{\beta\to 0}\sigbrb =0\] 
\begin{proof}
 Since the compactness and $\Gamma-$liminf inequality are readily obtained, let us focus on the $\Gamma-$limsup. For this, let $(v,\p)$  be such that $\mathcal{G}_{\rho,0}(v,\p)<+\infty$. Let then  $v_\beta:= v$ and 
\[\p_\beta(t):= \begin{cases}
                 0& \textrm{ for } t \in \lt(-\infty,-\frac{2}{\sqrt{\beta}}\rt]\\
		\sqrt{\beta}\p\lt(-\frac{1}{\sqrt{\beta}}\rt) \lt(t+\frac{1}{\sqrt{\beta}}\rt) +\p\lt(-\frac{1}{\sqrt{\beta}}\rt) & \textrm{ for } t\in \lt[-\frac{2}{\sqrt{\beta}},-\frac{1}{\sqrt{\beta}}\rt]\\
		\p(t) & \textrm{ for } t\in \lt[-\frac{1}{\sqrt{\beta}},\frac{1}{\sqrt{\beta}}\rt]\\
	      \sqrt{\beta}\lt(\pi-\p\lt(\frac{1}{\sqrt{\beta}}\rt)\rt) \lt(t-\frac{1}{\sqrt{\beta}}\rt) +\p\lt(\frac{1}{\sqrt{\beta}}\rt) & \textrm{ for } t\in \lt[\frac{1}{\sqrt{\beta}},\frac{2}{\sqrt{\beta}}\rt]\\
		\pi& \textrm{ for } t \in \lt[\frac{2}{\sqrt{\beta}}, +\infty\rt).
                \end{cases}
\] 
A simple computation then shows that 
\[\lt| \Eund(v_\beta,\p_\beta)- \mathcal{G}_{0}(v,\p)\rt|\le C \sqrt{\beta}.\]
\end{proof}

\end{theo}

\begin{rk}\rm
 From the proof, we see that $\sigbrb\le C \sqrt{\beta}$ which is exactly the scaling predicted in the physics literature \cite{Aochui,Timmermans,Mazets,barankov}.
\end{rk}

\subsection{Study of $\beta\to +\infty$ and symmetry breaking}
In this section we study the behavior of the limiting energy when $\beta\to +\infty$. We   prove that in this case, we recover the functional

\ben
    \mathcal{F}_\infty (\p) = \int_{\D}  \frac{\sigma_\infty}{\pi} |D\p|      
\een

derived in \cite{AftRoyo}, where $\sigma_\infty(x) := \frac{2\sqrt{2}}{3}\rho^{3/2}(x)$.\\

%

Let us prove that $\lim_{\beta \to \infty} \sigbrb=  {\ov \sigma}_\infty:= \frac{2\sqrt{2}}{3} $ with a rate of approximation of the order of $\beta^{-1/4}$ as predicted in the physical literature \cite{BVS}.
\begin{prop} \label{propsb}
 
\[ {\ov \sigma}_\infty \ge \sigbrb \ge {\ov \sigma}_\infty - \sqrt{2}\, C\, \beta^{-1/4}.\]
In particular, $\lim_{\beta\to +\infty}\sigbrb= {\ov \sigma}_\infty$. 
\end{prop}
\begin{proof}
 The upper bound is a consequence of \eqref{eqEm} with $m=0$. For the lower bound, we first notice that from Lemma \ref{decayinf}, we know that for every minimiser $v_\beta$ of $\sigbrb$, there holds $\inf v_\beta\le C \beta^{-1/4} $ so that as in the proof of Proposition \ref{away},
\[\sigbrb\ge \sqrt{2} \left(\frac{2}{3}-\inf v_\beta+ (\inf v_\beta)^3\left(\frac{1}{3}+ \frac{ \beta^{1/2}}{2\sqrt{2}}\right)\right)
\ge {\ov \sigma}_\infty - \sqrt{2}\, C\, \beta^{-1/4}.\]
\end{proof}

We then easily deduce the convergence of the full energy:
\begin{prop} \label{Gcvbeta}
 The $\Gamma$-limit in $L^1_{loc}(\D)$ as $\beta \to +\infty$ of $ \mathcal{F}_\beta$ is $\mathcal{F}_\infty$.
\end{prop}

Let us now concentrate on the harmonic potential $V=|x|^2$ and  let us study the minimisers of $\F_\infty$ under the mass constraint \eqref{mcfinal} 
to show the symmetry breaking. Let us point out again that since the functional $\mathcal{F}_\beta$ differs from $\mathcal{F}_\infty$ only by a (multiplicative) constant, 
the minimizers of the two functionals coincide. In particular, they do not depend on $\beta$. To prove symmetry breaking, we closely follow the ideas of  \cite[Cor. 1.3]{AftRoyo}, where such a
 result was derived for 
$n=2$. Let us first prove that the minimizer among radially symmetric sets is either  the centered ball or the outside annulus.
\begin{prop}
 Let $\alpha\in[0,1]$ and let $1\ge R_\alpha\ge 0$ be such that $\ds\int_{B_{\lambda R_\alpha}} \rho\, dx \, =\, \alpha$ then letting $f(\alpha):=\F_\infty(\pi \1_{B_{\lambda R_\alpha}})$,
\[
\min\{ \F_\infty(A) \; : \; A \textrm{ radially symmetric and satisfies \eqref{mcfinal}} \}=\min\lt(f(\alpha),f(1-\alpha)\rt) 
 \]
\end{prop}
\begin{proof}
 Let us first notice that $R_\alpha$ is determined by $\ds \alpha= \HH(\Sn)\lambda^{n+2} \int_0^{R_\alpha} (1-r^2) r^{n-1} dr$, so that 
\begin{equation}\label{derivR}
R_\alpha'=\lt(\HH(\Sn)\lambda^{n+2} (1-R_\alpha^2)R_\alpha^{n-1}\rt)^{-1}\end{equation}
where by a slight abuse of notation we identified $R_\alpha$ with the function $\alpha\to R_\alpha$. A simple computation shows that for $\alpha\in(0,1)$,
\begin{equation}\label{enerball}
 f(\alpha)=\frac{2\sqrt{2}}{3} \HH(\Sn)\lambda^{n+2} R_\alpha^{n-1} (1-R_\alpha^2)^{3/2}
\end{equation}
 and $f(0)=f(1)=0$. It then follows from \eqref{derivR} that for $\alpha\in (0,1)$,
\[f''(\alpha)=-\frac{2\sqrt{2}}{3\HH(\Sn) \lambda^{n+2}} (1-R_\alpha^2)^{-5/2} R_\alpha^{-(n+1)}\lt((n-1)(1-R_\alpha^2) +3 R_\alpha^2\rt)<0\]
and thus $f$ is strictly concave \footnote{notice that for $n=1$, $f$ is discontinuous at $0$ but is still strictly concave since $f\ge f(0)$ and $f$ is strictly concave in $(0,1)$}.\\

Let now $A(R_1,R_2):=\{\lambda R_1< |x|\le \lambda R_2\}$ be an annulus with $0<R_1<R_2<1$ and $\ds \int_{A(R_1,R_2)} \rho \, dx =\alpha$, then letting
\[\beta_1:= \int_{B_{\lambda R_1}} \rho\, dx \qquad \textrm{ and } \qquad \beta_2:= \int_{\D\backslash B_{\lambda R_2}} \rho\, dx\]
we have $\beta_1+\beta_2=1-\alpha$ and 
\[f_1(\beta_1):=\F_\infty\lt(\pi \1_{A(R_1,R_2)}\rt)= f(\beta_1)+f(\beta_1+\alpha)\]
is a strictly concave function of $\beta_1$ and thus attains its minimum for $\beta_1=0$ or $\beta_1=1-\alpha$. This proves that 
\[\inf_{R_1,R_2} \F_{\infty} (\pi \1_{A(R_1,R_2)})=\min( f(\alpha),f(1-\alpha)).\]
As in \cite{AftRoyo}, by induction it implies that any union of $m\in \N$ annuli has energy larger than $\min( f(\alpha),f(1-\alpha))$ which in turn by approximation
 implies that any radially symmetric set has energy at least $\min( f(\alpha),f(1-\alpha))$.  
\end{proof}
In order to show symmetry breaking it is thus enough to construct a non radially symmetric set with energy smaller than $\min(f(\alpha),f(1-\alpha))$.
\begin{prop}\label{brissym}
 Let $n=1,2$ or $3$ then there exists $\alpha_0\in (0,1/2)$ such that if $\alpha\in(\alpha_0,1-\alpha_0)$, the minimizers of $\F_\infty$ under the mass constraint \eqref{mcfinal} are not radially symmetric.
\end{prop}
\begin{proof}
 For $n=2$, the proof is already given in \cite[Cor. 1.3]{AftRoyo}. For $n=1$, consider the interval $A_\alpha:=(-\lambda, t_\alpha]$ where
 $t_\alpha$ is chosen so that $\ds \int_{-\lambda}^{t_\alpha} (\lambda^2-|x|^2) dx =\alpha$. We then have $\F_\infty(\pi \1_{A_{\alpha}})=\frac{2\sqrt{2}}{3} \sqrt{\lambda^2-t_\alpha^2}$. By continuity of
 $\F_\infty(\pi \1_{A_{\alpha}})$ and $f$ with respect to $\alpha$, it is enough showing that $\F_\infty(\pi \1_{A_{\alpha}})<\min(f(\alpha),f(1-\alpha))$
for $\alpha=1/2$ so that on the one hand  $t_\alpha=0$ and $\F_\infty(\pi \1_{A_{\alpha}})=\frac{2\sqrt{2}}{3}\lambda^3$ and on the other hand $f(1/2)=\frac{2\sqrt{2}}{3} 2\lambda^3 (1-R_\alpha^2)^{3/2}$. It is thus enough checking that
\[2(1-R_\alpha^2)^{3/2}>1.\]
We find that $R_\alpha \approx 0.35$  and thus $2(1-R_\alpha^2)^{3/2}\approx 1.65 >1$.\\
For $n=3$, let us consider in cylindrical coordinates the set $A_\alpha:=\{(r \exp(i\theta),z) \; : \; r\in(0, \lambda), \, \theta\in(0,\theta_\alpha), \, z\in(-\lambda,\lambda)\}$ where $\theta_\alpha$ is such that \eqref{mcfinal} is satisfied. It is readily seen that $\mathcal{F}_\infty(\pi \1_{A_\alpha})=\frac{2\sqrt{2}}{3} \frac{2\pi}{5} \lambda^5 $
 (notice that it does not depend on $\alpha$). As for $n=1$, it is enough to compare it with $f(1/2)=\frac{2\sqrt{2}}{3} 4\pi \lambda^5 R_\alpha^2 (1-R_\alpha^2)^{3/2}$ so that we are left to check that
\[10 R_\alpha^2 (1-R_\alpha^2)^{3/2} >1.\]
We find $R_{\alpha}\approx 0.64$ and thus $10 R_\alpha^2 (1-R_\alpha^2)^{3/2}\approx 1.86 >1$.

\end{proof}

 Using the properties of $\Gamma$-convergence  we derive the analogous result
for the minimisers of $\Feb$ with  $\eps$ small enough:

\begin{cor}
Let $n=1,2$ or $3$ and $V=|x|^2$.  There exists $\delta_0 \in (0,1/2) $ such that for $\alpha_1 \in [\delta_0 , 1-\delta_0 ]$  and $\beta > 0$,  there exists $\eps(\beta)>0$ such that for $0<\eps<\eps(\beta)$, 
the minimisers of $\Feb$ under the constraint \eqref{massvphi} are not radially symmetric.
 \end{cor}

\section*{Acknowledgment}
We thank P. Bella for pointing out the paper \cite{CherMur} and for suggesting the equipartition of the energy for the optimal profile. We are very grateful to B. Van Schaeybroeck for 
drawing our attention to the physics literature concerning the computation of the surface tension. 
We thank B. Merlet for pointing out a small mistake in Proposition \ref{propertymin}. The  authors wish to warmly thank the hospitality of the
`Max Planck Institut f\"ur Mathematik' in Leipzig, where this work was started.
 \bibliography{BEC}
\end{document}